\documentclass[11pt]{article}
\usepackage{geometry,amsmath,amsfonts,amssymb,amsthm,color,mathrsfs,indentfirst}
\usepackage[colorlinks=true]{hyperref}
\numberwithin{equation}{section}
\newtheorem{theorem}{Theorem}[section]
\newtheorem{lemma}{Lemma}[section]

\allowdisplaybreaks

\begin{document}
\vspace{5cm}


\title{{Global classical solutions for a class of reaction-diffusion
system with density-suppressed motility}}
\author{Wenbin Lyu \thanks{%
School of Mathematical Sciences, Shanxi University, Taiyuan 030006, P.R. China; lvwenbin@sxu.edu.cn. }  \and Zhi-An Wang
\thanks{%
Department of Applied Mathematics, The Hong Kong Polytechnic University, Hong Kong; mawza@polyu.edu.hk. }}
\date{}
\maketitle

%
%

\vspace{0.5cm}

{\noindent\bf Abstract.}
This paper is concerned with a class of reaction-diffusion system with density-suppressed motility
\begin{equation*}
\begin{cases}
u_{t}=\Delta(\gamma(v) u)+\alpha u F(w), & x \in \Omega, \quad t>0, \\
v_{t}=D \Delta v+u-v, & x \in \Omega, \quad t>0, \\
w_{t}=\Delta w-u F(w), & x \in \Omega, \quad t>0,
\end{cases}
\end{equation*}
under homogeneous Neumann boundary conditions in a smooth bounded domain $\Omega\subset \mathbb{R}^n~(n\leq 2)$, where $\alpha>0$ and $D>0$ are constants. The random motility function $\gamma$ satisfies
\begin{equation*}
\gamma\in C^3((0,+\infty)),\ \gamma>0,\ \gamma'<0\,\ \text{on}\,\ (0,+\infty) \ \ \text{and}\ \ \lim_{v\rightarrow+\infty}\gamma(v)=0.
\end{equation*}
The intake rate function $F$ satisfies
\begin{equation*}
F\in C^1([0,+\infty)),\,F(0)=0\,\ \text{and}\ \,F>0\,\ \text{on}\,\ (0,+\infty).
\end{equation*}
We show that the above system admits a unique global classical solution for all non-negative initial data
$$
u_0\in C^0(\overline{\Omega}),\,v_0\in W^{1,\infty}(\Omega),\,w_0\in W^{1,\infty}(\Omega).
$$
 Moreover, if there exist $k>0$ and $\overline{v}>0$ such that
\begin{equation*}
\inf_{v>\overline{v}}v^k\gamma(v)>0,
\end{equation*}
then the global solution is bounded uniformly in time.
\vspace{0.2cm}

{\noindent\bf Keywords}: Reaction-diffusion system, Density-suppressed motility, Global existence,  Boundedness.
\vspace{0.2cm}

{\noindent\bf 2020 Mathematics Subject Classification.}: 35A01, 35B45, 35K51, 35Q92.


\section{Introduction and main results}

To explain the strip pattern formation observed in the experiment of \cite{L.2011S} induced by the ``self-trapping mechanism'', the following three-component reaction-diffusion system  with density-dependent motility was proposed in \cite{L.2011S}
\begin{equation}\label{dsm}
\begin{cases}
u_{t}=\Delta(\gamma(v) u)+\frac{\alpha w^2 u}{w^2+\lambda} , & x \in \Omega, \quad t>0, \\
v_{t}=D \Delta v+u-v, & x \in \Omega, \quad t>0, \\
w_{t}=\Delta w-\frac{\alpha w^2 u}{w^2+\lambda}, & x \in \Omega, \quad t>0,
\end{cases}
\end{equation}
where $u(x,t),v(x,t),w(x,t)$ denote the bacterial cell density, concentration of acyl-homoserine lactone (AHL) and nutrient density, respectively; $\alpha,\lambda,D>0$ are constants and $\Omega$ is  a bounded domain in $\mathbb{R}^n$. The first equation of \eqref{dsm} describes the random motion of bacterial cells with AHL-density dependent motility coefficient $\gamma(v)$, and cell growth due to the nutrient intake. The second equation of \eqref{dsm} describes the  diffusion, production and turnover of AHL, while the third equation gives the dynamics of the nutrient with diffusion and consumption. Simultaneously  a simplified two-component system was discussed in the supplemental material of \cite{L.2011S} and formally analyzed in \cite{F.2012PRL}:
\begin{equation}\label{OR-2}
\begin{cases}
u_t=\Delta(\gamma(v)u)+\mu u(1-u),&x\in\Omega, t>0,\\
v_t=D\Delta v+u-v,& x\in \Omega, t>0,
\end{cases}
\end{equation}
where the decay of bacterial cells at high density was used to approximate the nutrient depletion effect. A striking feature of systems \eqref{dsm} and \eqref{OR-2} is that the cell diffusion rate depends on a motility function $\gamma(v)$ satisfying $\gamma'(v)<0$, which takes into account the repressive effect of AHL concentration on the cell motility (cf. \cite{L.2011S}). The density-suppressed motility mechanism  has also been used to model other biological processes, preytaxis \cite{KO-1987,J-W2020EJAM} and chemotaxis \cite{KS-1971-JTB2, Z.A.W.2020PRE}.
 From the expansion
$$\Delta (\gamma(v)u)=\nabla \cdot (\gamma(v)\nabla u+u\gamma'(v) \nabla v)=\gamma(v) \Delta u+2 \gamma'(v)\nabla v\cdot\nabla u+u\gamma''(v)|\nabla v|^2+u\gamma'(v)\Delta v,$$
we see that the nonlinear diffusion rate function $\gamma(v)$ not only contributes a cross-diffusion structure but also
renders a possible diffusion degeneracy (i.e., $\gamma(v) \to 0$ as $v \to +\infty$). Therefore many conventional methods are inapplicable and the analysis of \eqref{dsm} or \eqref{OR-2} is very delicate. The progresses were not made to the system \eqref{OR-2} until recently with homogeneous Neumann boundary conditions in a smooth bounded domain $\Omega \subset \mathbb{R}^n$. The existing results on \eqref{OR-2} can be classified into two cases:  $\mu>0$ and $\mu=0$, to be recalled below.

{\bf Case of $\mu>0$}. When the motility function $\gamma(v)$ satisfies
$\gamma(v)\in C^3([0,\infty)),\ \gamma(v)>0\ \ \mathrm{and}\ \ \gamma'(v)<0\ \mathrm{on}\ [0,\infty)$, $\lim\limits_{v \to \infty}\gamma(v)=0$ and $\lim\limits_{v \to \infty}\frac{\gamma'(v)}{\gamma(v)}$ exists, it was first shown in \cite{J-K-W2018} that the system \eqref{OR-2} has a unique global classical solution in two dimensions ($n=2$) which globally asymptotically converges to the equilibrium $(1,1)$  if
$\mu>\frac{K_0}{16}$ with $K_0=\max\limits_{0\leq v \leq \infty}\frac{|\gamma'(v)|^2}{\gamma(v)}$. The global existence result and large-time behavior were subsequently extended to higher dimensions ($n\geq 3$) for large $\mu>0$ in \cite{W-W2019JMP} and \cite{L-X2019}. The condition that $\lim\limits_{v \to \infty}\frac{\gamma'(v)}{\gamma(v)}$ exists imposed in \cite{J-K-W2018} was recently removed in \cite{F-J2020JDE,J-W2020PRE} for the global existence of solutions. On the other hand, for small $\mu>0$, the existence/{\color{black}nonexistence} of nonconstant steady states  of \eqref{OR-2} was rigorously established in \cite{M-P-W2020PD, X-W2020PRE} in appropriate parameter regimes. Some other works with generalized logistic source or indirect production of chemical signals can be found in \cite{W.B.L.2020NA,L-W2020JMAA,L-W2020ZAMP,L-W2020P}.

{\bf Case of $\mu=0$}. Turning to the case $\mu=0$, it was first shown in \cite{T-W2017MMMAS} that globally bounded solutions exist in two dimensions if the motility function $\gamma(v)$ has both positive lower and upper bounds. If $\gamma(v)=\frac{c_0}{v^k}$, it was proved in \cite{Y-K2017AAM} that global bounded solutions exist in all dimensions provided that $c_0>0$ is small. The smallness of $c_0$ is later removed in \cite{A-Y2019N} for the parabolic-elliptic case model (i.e. the second equation of \eqref{OR-2} is replaced by $D\Delta v+u-v=0$) with $0<k<\frac{2}{n-2}$. The global existence of weak solutions of \eqref{OR-2} with large initial data was established in \cite{D-K-T-Y2019NARWA} for $\gamma(v)=\frac{1}{c+v^k}$ with $c\geq 0, k>0$ and $1\leq n\leq 3$. When $\gamma(v)=e^{-\chi v}$, a critical mass phenomenon was identified in \cite{J-W2020PAMS}:  if $n=2$, there is a critical number $m=4\pi/\chi>0$ such that the solution of \eqref{OR-2} with $D=1$ may blow up if the initial cell mass $\|u_0\|_{L^1(\Omega)}>m$, while global bounded solutions exist if $\|u_0\|_{L^1(\Omega)}<m$. This result was further refined in \cite{F-J2020preprint,F-J2020JDE} showing that the blowup occurs at the infinity time. Recently global weak solutions of \eqref{OR-2} in all dimensions and the blow-up of solutions of \eqref{OR-2} in two dimensions were investigated in \cite{Burger}.

Except the studies on the bounded domain with zero Neumman boundary conditions, there are some results obtained in the whole space: when $\gamma(v)$ is a piecewise constant function, the dynamics of discontinuity interface was studied in \cite{SIK-EJAP-2019} and  discontinuous traveling wave solutions of \eqref{OR-2} with $\mu>0$ were constructed in \cite{Lui}; the existence of smooth traveling wave solutions of \eqref{OR-2} with $\mu>0$ and a particular motility function $\gamma(v)=1/(1+v)^m (m>0)$ was recently shown in \cite{Li-Wang-2020}.

Compared to the abundant results recently obtained for the system \eqref{OR-2} as recalled above, the progress made to the three-component system \eqref{dsm} is very limited. The purpose of this paper is explore the global well-posedness of the following system
\begin{equation}\label{1}
\begin{cases}
u_{t}=\Delta(\gamma(v) u)+\alpha u F(w), & x \in \Omega, \quad t>0, \\
v_{t}=D \Delta v+u-v, & x \in \Omega, \quad t>0, \\
w_{t}=\Delta w-u F(w), & x \in \Omega, \quad t>0, \\
\frac{\partial u}{\partial \nu}=\frac{\partial v}{\partial \nu}=\frac{\partial w}{\partial \nu}=0, & x \in \partial \Omega, \quad t>0, \\
(u, v, w)(x, 0)=\left(u_{0}, v_{0}, w_{0}\right)(x), & x \in \Omega,
\end{cases}
\end{equation}
with constants $\alpha>0$ and $D>0$, where the system \eqref{dsm} is a special case of \eqref{1} with $F(w)=\frac{w^2}{w^2+\lambda}$. By postulating that
\begin{equation}\label{assumption1}
\gamma(v)\in C^3([0,+\infty))\, \ \text{and}\,\ 0<\gamma_1\leq\gamma(v)\leq\gamma_2,\,|\gamma'(v)|<\eta\, \ \text{on}\,\ [0,+\infty)
\end{equation}
where $\gamma_1, \gamma_2$ are positive constants, and
$$F\in C^1([0,+\infty)),\,F(0)=0\,\ \text{and}\,\ F(w)>0,\,F'(w)>0\,\ \text{on}\,\ (0,+\infty),$$
a recent work \cite{J-S-W2020JDE} showed that the problem \eqref{1} admits a global classical solution $(u, v, w)$ which asymptotically converges to $(u_* , u _*, 0)$ in $L^\infty$ with $u_*=\frac{1}{|\Omega|}\left(\|u_0\|_{L^1}+\alpha\|w_0\|_{L^1}\right)$ if $D>0$ is suitably large. The main approaches employed in \cite{J-S-W2020JDE} to establish the global classical solutions with uniform-in-time bounds are based on the method of energy estimates and Moser iteration by fully capturing the diffusive dissipation of $u$ with the assumption that $\gamma(v)$ has a positive lower bound. The assumption \eqref{assumption1} bypasses the possible diffusion degeneracy/singularity and rules out a large class of functions such as $\gamma(v)=\frac{c_0}{v^k} (c_0, k>0)$ and $\gamma(v)=e^{-\chi v} (\chi>0)$  widely studied in the existing works as recalled above. The goal of this paper is to remove this essential restriction imposed in \eqref{assumption1} and establish the global well-posedness of solutions to \eqref{1}. Roughly speaking,  under the following structural assumptions on $\gamma(v)$ and $F(v)$:
\begin{equation}\label{4}
\gamma(v)\in C^3([0,+\infty)),\ \gamma(v)>0,\ \gamma'(v)<0\ \ \text{on}\ \ (0,+\infty), \ \text{and}\,\ \lim_{v\rightarrow+\infty}\gamma(v)=0,
\end{equation}
and
\begin{equation}\label{5}
F\in C^1([0,+\infty)),\,\ F(0)=0\,\ \text{and}\,\ F>0\ \ \text{on}\ \ (0,+\infty),
\end{equation}
then for any initial data $(u_0, v_0, w_0)$ satisfying
\begin{equation}\label{12}
\begin{gathered}
u_0\in C^0(\overline{\Omega}),\,v_0\in W^{1,\infty}(\Omega),\,w_0\in W^{1,\infty}(\Omega),\\
u_0\geq0,\,v_0\geq c_0>0,\,w_0\geq0\,\text{and}\,\,u_0\not\equiv0,
\end{gathered}
\end{equation}
with some constant $c_0>0$, we show the problem \eqref{1} admits a unique global classical solution in two dimensions. Moreover
if there exist $k>0$ and $\bar{v}>0$ such that
\begin{equation}\label{23}
\inf_{v>\bar{v}}v^k\gamma(v)>0
\end{equation}
the solution is uniformly bounded in time.

Our main results are precisely stated as follows.

\begin{theorem}\label{th1}
Let $\Omega\subset \mathbb{R}^n(n\leq 2)$ be a bounded domain with smooth boundary. Assume that the conditions \eqref{4} and \eqref{5} hold. Then for any initial data $(u_0,v_0,w_0)$ satisfying the condition \eqref{12}, there exists a triple
$(u,v,w)$ of non-negative functions
$$(u,v,w)\in \left[C^0(\overline{\Omega}\times[0,+\infty))\cap C^{2,1}(\overline{\Omega}\times(0,+\infty))\right]^3$$
which solves \eqref{1} in the classical sense. Moreover, if the motility function $\gamma$ satisfies the condition \eqref{23}, then the global solution is uniformly bounded in time, that is  there exists a constant $C>0$ such that
\begin{equation*}
\|u(\cdot,t)\|_{L^\infty}+\|v(\cdot,t)\|_{W^{1,\infty}}+\|w(\cdot,t)\|_{W^{1,\infty}}\leq C \ \ \text{for all}\ \ t>0.
\end{equation*}
\end{theorem}

The key of proving Theorem \ref{th1} is to derive that $v$ has a positive lower bound to rule out the diffusion singularity and has an upper bound to exclude the diffusion degeneracy (see section 3.3.) The positive lower bound of $v$ can be obtained easily by showing that $\int_\Omega udx$ has a positive lower bound along with a nice result of \cite{K.F.2015JMAA}. The crucial step is to show that $v$ has an upper bound.  Inspired by an idea from the work \cite{F-J2020JDE}, we construct an auxiliary problem and use the maximum principle for the inverse operator $(I-D\Delta)^{-1}$ to derive an upper bound of $v$ through the auxiliary problem.

The rest of this paper is organized as follows. Section \ref{sec2} is devoted to the local existence of solutions and extensibility of \eqref{1}. With some important inequalities which will be used frequently, we derive a priori estimates of solutions for the system \eqref{1} in section \ref{sec3}. Finally, we prove Theorem \ref{th1} in section \ref{sec4}.

\section{Preliminaries}\label{sec2}
In this section, we present some basic results and facts, including local existence and extensibility criterion of classical solutions as well as some frequently used well-known inequalities.

The existence of local solutions and extensibility criterion for the system \eqref{1} can be obtained by Amann's theorem (cf. \cite{H.A.1993}) or fixed point theorem (cf. \cite{J-K-W2018}). Below, we only state the local existence result without proof.

\begin{lemma}[Local existence]\label{lm3}
Let $\Omega\subset \mathbb{R}^n$ be a bounded domain with smooth boundary. If the initial data satisfy the condition \eqref{12}, then there exist a constant $T_{max}\in(0,\infty]$ and a triple $(u,v,w)$ of non-negative functions
$$(u,v,w)\in \left[C^0(\overline{\Omega}\times[0,T_{max}))\cap C^{2,1}(\overline{\Omega}\times(0,T_{max}))\right]^3,$$
which solves \eqref{1} in the classical sense in $\Omega\times(0,T_{max})$. Moreover, if $T_{max}<+\infty$, then
\begin{equation*}
\limsup\limits_{t\nearrow T_{max}}(\|u(\cdot,t)\|_{L^\infty}+\|v(\cdot,t)\|_{W^{1,\infty}}+\|w(\cdot,t)\|_{W^{1,\infty}})=\infty.
\end{equation*}
\end{lemma}

Next, we recall some well-known results which will be used later frequently. The first one is an ODE inequality \cite{S-S-W2014}.

\begin{lemma}\label{lm16}
Let $T_{max}>0$, $\tau\in(0,T_{max})$, $a>0$ and $b>0$. Suppose that $y:[0,T_{max})\rightarrow[0,\infty)$ is absolutely
continuous and satisfies
$$y'(t)+ay(t)\leq h(t)\ \ \text{for all}\ \ t\in(0,T_{max})$$
with some nonnegative function $h\in L^1_{loc}([0,T_{max}))$ fulfilling
$$\int_t^{t+\tau}h(s)ds\leq b\ \ \text{for all}\ \ t\in[0,T_{max}-\tau).$$
Then it follows that
$$y(t)\leq\max\left\{y(0)+b,\frac{b}{a\tau}+2b\right\}\ \ \text{for all}\ \ t\in[0,T_{max}).$$
\end{lemma}

Below is an uniform Gr\"{o}nwall inequality \cite{R.T.1988} which can help us to derive the uniform-in-time estimates of solutions.

\begin{lemma}\label{lm9}
Let $T_{max}>0$, $\tau\in(0,T_{max})$. Suppose that $a,b,y$ are three positive locally integrable functions on $(0,T_{max})$ such that $y'$ is locally integrable on $(0, T_{max})$ and the following inequalities are satisfied:
$$y'(t) \leq a(t)y(t)+b(t)\ \ \text{for all}\ \ t\in(0,T_{max})$$
as well as
$$\int_{t}^{t+\tau} a\leq a_{1}, \ \ \int_{t}^{t+\tau} b \leq a_{2}, \ \  \int_{t}^{t+\tau} y \leq a_{3}\ \ \text{for all}\ \ t\in[0, T_{max}-\tau),$$
where $a_{i}(i=1,2,3)$ are positive constants. Then
$$y(t) \leq\left(\frac{a_{3}}{\tau}+a_{2}\right) e^{a_{1}}\ \ \text{for all}\ \ t\in[\tau, T_{max}).$$

\end{lemma}

The third one is about the regularity of solutions to the linear parabolic equation and the proof can be found in \cite{K-S2008}.

\begin{lemma}\label{lm10}
Assume that $\Omega\subset\mathbb{R}^n$ is a bounded domain with smooth boundary. Suppose that
$y(x,t)\in C^{2,1}(\bar{\Omega}\times(0,T_{max}))$ is the solution of
\begin{equation*}
 \begin{cases}
 y_{t}=\Delta y-y+\phi(x,t),& x\in\Omega,t\in(0,T_{max}),\\
 \frac{\partial y}{\partial \nu}=0,&x\in\partial\Omega,t\in(0,T_{max}),\\
 y(x,t)=y_0(x) \in C^0 (\bar{\Omega}),
 \end{cases}
\end{equation*}
where $\phi(x,t)\in L^{\infty}((0,T_{max});L^p(\Omega))$. Then there exists a constant $C>0$ such that
\begin{equation*}
  \|y(\cdot,t)\|_{W^{1,q}}\leq C \ \ \text{for all}\ \ t\in(0,T_{max})
\end{equation*}
with
\begin{equation*}
 q\in
  \begin{cases}
  [1,\frac{np}{n-p}), &\text{if}~p\leq n,\\
  [1,\infty],&\text{if}~p> n.
  \end{cases}
\end{equation*}

\end{lemma}

\section{A priori estimates}\label{sec3}
This section is devoted to deriving a priori estimates of solutions for the system \eqref{1}, so that the global existence of solutions can be obtained by the extensibility criterion. We will proceed in several steps. In the following, we shall use $C_i (i=1,2, \cdots)$ to denote a generic positive constant which may vary in the context.

\subsection{The boundedness of \texorpdfstring{$u$}{u} in \texorpdfstring{$L^1(\Omega)$}{L1}}

A basic property of solutions is the uniform-in-time $L^1$ boundedness of $u$ due to the special structure  of the system \eqref{1}.

\begin{lemma}\label{lm12}
Let $(u,v,w)$ be a solution of \eqref{1}. Then there exist constants $c,C>0$ such that
\begin{equation}\label{13}
c\leq\int_\Omega u\leq C\ \ \text{for all}\ \ t\in(0,T_{max}).
\end{equation}
\end{lemma}

\begin{proof}
Integrating the first equation of \eqref{1} over $\Omega$ with the boundary conditions, we have
$\frac{d}{dt}\int_\Omega u\geq 0$
which implies
$$\int_\Omega u\geq \int_\Omega u_0.$$
We multiply the third equation of \eqref{1} by $\alpha$ and add the resulting equation to the first equation of \eqref{1}. Then integrating the result over $\Omega$ by parts along with the boundary conditions, we get
$$\frac{d}{dt}\left(\int_\Omega u+\alpha\int_\Omega w\right)=0$$
which yields
$$\int_\Omega u+\alpha\int_\Omega w=\int_\Omega u_0+\alpha\int_\Omega w_0.$$
Then, the non-negativity of $u$ and $w$ yields \eqref{13}.
\end{proof}

\subsection{The upper bound of \texorpdfstring{$w$}{w}}

The following lemma concerns the upper bound of $w$ which is an immediate consequence of the maximum principle (see \cite{J-S-W2020JDE}).

\begin{lemma}\label{lm7}
Let $(u,v,w)$ be a solution of \eqref{1}. We can find a constant $C>0$ such that
\begin{equation*}
w\leq C\ \ \text{for any}\ \ (x,t)\in\Omega\times(0,T_{max}).
\end{equation*}
\end{lemma}

\subsection{The lower and upper bounds of \texorpdfstring{$v$}{v}}

The following lemma is vital for us to rule out the possible singularity of $\gamma(v)$ near $v=0$. The mass inequality \eqref{13} plays a key role in the proof of this lemma. The proof can be found in \cite{K.F.2015JMAA}.

\begin{lemma}\label{lm2}
Let $(u,v,w)$ be a solution of \eqref{1}. Then there exists a constant $C>0$ fulfilling
$$v\geq C\ \ \text{for any}\ \ (x,t)\in\Omega\times(0,T_{max}).$$
\end{lemma}

Motivated from the paper \cite{F-J2020preprint,F-J2020JDE}, next, we derive the upper bound of $v$, which is a key to preclude the degeneracy of diffusion.

Let us introduce an auxiliary function $g$ which satisfies the following equation
\begin{equation}\label{24}
\begin{cases}
-D\Delta g+g=u,&x\in\Omega,\,t\in[0,T_{max}),\\
\frac{\partial g}{\partial \nu}=0,&x\in\partial\Omega,\,t\in[0,T_{max}),\\
g(x,0)=g_0(x)\geq0,&x\in\Omega,
\end{cases}
\end{equation}
where $u$ is the solution of \eqref{1} obtained in Lemma \ref{lm3}. Then, $g$ is non-negative since $u\geq0$ and has the following basic properties. The first property states that the $L^q$ norm of $g$ can be controlled by the $L^1$ norm of $u$ (cf. \cite{B-S1973JMSJ}).

\begin{lemma}\label{lm15}
Let $u\in C(\overline{\Omega})$ be a non-negative function such that $\int_\Omega u>0$. If $g\in C^2(\overline{\Omega})$ is a solution of the system \eqref{24}, then for any $q$ satisfying $1\leq q<\frac{n}{(n-2)_+}$ there exists a constant $C>0$ such that
$$\|g\|_{L^q}\leq C\|u\|_{L^1}.$$
\end{lemma}

The second property tells us that $g$ satisfies a simple inequality.

\begin{lemma}\label{lm8}
Let $(u,v,w)$ be a solution of \eqref{1} and $g$ satisfies \eqref{24}. Then for all
$(x,t)\in\Omega\times(0,T_{max})$, we have
\begin{equation}\label{6}
g_t+\frac{1}{D}\gamma(v)u=\frac{1}{D}(I-D\Delta)^{-1}[\gamma(v)u]+\alpha (I-D\Delta)^{-1}[u F(w)].
\end{equation}
Moreover, there exists a constant $C>0$ such that
\begin{equation}\label{7}
g_t\leq Cg \ \ \text{for any}\ \ (x,t)\in\Omega\times(0,T_{max}).
\end{equation}
\end{lemma}

\begin{proof}
The first equation of \eqref{1} can be rewritten as
$$u_t=-\frac{1}{D}(I-D\Delta)[\gamma(v)u]+\alpha u F(w)+\frac{1}{D}\gamma(v)u.$$
Taking the operator $(I-D\Delta)^{-1}$ on both side of the above equation and noticing the definition of $g$, we can get \eqref{6} directly.

Now we prove \eqref{7}. According to the non-increasing property of $\gamma$ and Lemma \ref{lm2}, there exists a constant $C_1>0$ such that
$$\gamma(v)\leq C_1.$$
Noticing Lemma \ref{lm3}, Lemma \ref{lm7} and the smoothing property of $F$, we get a constant $C_2>0$ such that
$$F(w)\leq C_2.$$
Owing to the nonnegative of $u$, it holds that
$$\gamma(v)u\geq0.$$
Recall \eqref{24}. Then by the comparison principle for elliptic equations, we have
\begin{equation*}
\frac{1}{D}(I-D\Delta)^{-1}[\gamma(v)u]+\alpha (I-D\Delta)^{-1}[u F(w)]\leq\left(\frac{C_1}{D}+\alpha C_2\right)g,
\end{equation*}
which, combined with \eqref{6}, implies that
$$g_t\leq\left(\frac{C_1}{D}+\alpha C_2\right)g.$$
This finishes the proof.
\end{proof}

With the help of Lemma \ref{lm8} and the standard comparison principle for parabolic equations, we shall derive the upper bound of $v$.

\begin{lemma}\label{lm1}
Let $(u,v,w)$ be a solution of \eqref{1} and $g$ satisfies \eqref{24}. Then there is a constant $C>0$ such that
\begin{equation*}
v\leq C(g+1)\ \ \text{for any}\ \ (x,t)\in\Omega\times(0,T_{max}).
\end{equation*}
Moreover, if $T_{max}<+\infty$, there exists $C_0>0$ such that
$$v\leq C_0\ \ \text{for any}\ \ (x,t)\in\Omega\times(0,T_{max}).$$

\end{lemma}

\begin{proof}
With the hypothesis \eqref{4}, we can choose a constant $C_1\geq 0$ large enough such that
$$0<\gamma(C_1)<D.$$
Let
$$\Gamma(s):=\frac{1}{D}\int_{C_1}^s\gamma(x)dx\quad\text{for all }s\geq0,$$
which gives
\begin{align*}
\gamma(v)u=&\gamma(v)(v_t-D\Delta v+v)\\
=&D\Gamma_t(v)-D^2\Delta\Gamma(v)+D\gamma'(v)|\nabla v|^2+\gamma(v)v.
\end{align*}
This, combined with Lemma \ref{lm8}, implies
\begin{equation}\label{8}
\begin{aligned}
v_t-D\Delta v+v=&-D\Delta g+g\\
=&g_t-D\Delta g+g+\frac{1}{D}\gamma(v)u\\
&\quad-\frac{1}{D}(I-D\Delta)^{-1}[\gamma(v)u]-\alpha (I-D\Delta)^{-1}[u F(w)]\\
=&(g+\Gamma(v))_t-D\Delta(g+\Gamma(v))+(g+\Gamma(v))+\gamma'(v)|\nabla v|^2+\frac{1}{D}\gamma(v)v-\Gamma(v)\\
&\quad-\frac{1}{D}(I-D\Delta)^{-1}[\gamma(v)u]-\alpha (I-D\Delta)^{-1}[u F(w)].
\end{aligned}
\end{equation}
Now, we estimate the terms on the right hand side of \eqref{8}. In view of the monotone decreasing property of $\gamma$, Lemma \ref{lm2} and the definition of $\Gamma$, we see that there exists a constant $C_2>0$ such that
\begin{align*}
\frac{1}{D}\gamma(v)v-\Gamma(v)=&\frac{1}{D}\gamma(v)v+\frac{1}{D}\int^{C_1}_v\gamma(x)dx\\
\leq&\frac{1}{D}\left[\gamma(v)v+\gamma(v)(C_1-v)\right]\\
=&\frac{C_1}{D}\gamma(v)\leq\frac{C_2}{D}\quad\text{for }0\leq v\leq C_1
\end{align*}
or otherwise
\begin{align*}
\frac{1}{D}\gamma(v)v-\Gamma(v)=&\frac{1}{D}\gamma(v)v-\frac{1}{D}\int_{C_1}^v\gamma(x)dx\\
\leq&\frac{1}{D}\left[\gamma(v)v-\gamma(v)(v-C_1)\right]\\
=&\frac{C_1}{D}\gamma(v)\leq\frac{C_2}{D}\quad\text{for }v\geq C_1.
\end{align*}
Due to the non-negativity of $-\gamma'(v)$, $\gamma(v)u$ as well as $u F(w)$ and the comparison principle for elliptic equations, we get from \eqref{8}
$$v_t-D\Delta v+v\leq(g+\Gamma(v))_t-D\Delta(g+\Gamma(v))+(g+\Gamma(v))+\frac{C_2}{D}.$$
Noticing the initial data, we can choose a constant $C_3>0$ large enough such that $\frac{C_2}{D}\leq C_3$ and
$$v_0\leq g_0+\Gamma(v_0)+C_3.$$
Hence, the comparison principle for parabolic equations gives that
$$v\leq g+\Gamma(v)+C_3,$$
which along with the fact
$$\Gamma(v)\leq\frac{\gamma(C_1)}{D}v,$$
implies
$$v\leq \frac{1}{1-\frac{\gamma(C_1)}{D}}(g+C_3).$$
With the aid of Lemma \ref{lm8}, if $T_{max}<+\infty$, then there exists a constant $C_4>0$ such that
$$v\leq C_4.$$
Hence we complete the proof of this lemma.
\end{proof}

Note the upper bound derived in Lemma \ref{lm1} may depend on $T_{max}$, see \eqref{7}. The following lemma asserts the upper bound of $v$ which is independent of $T_{max}$ under additional condition \eqref{23}.

\begin{lemma}
Let $(u,v,w)$ be a solution of \eqref{1}. If the motility function $\gamma$ satisfies the condition \eqref{23}, then there exists a constant $C>0$ independent of $T_{max}$ such that
\begin{equation*}
v\leq C\ \ \text{for any}\ \ (x,t)\in\Omega\times(0,T_{max}).
\end{equation*}

\end{lemma}

\begin{proof}
We can rewrite the first equation of \eqref{1} as
$$\left((I-D\Delta)g\right)_t+\frac{1}{D}(I-D\Delta)(\gamma(v)u)=\frac{1}{D}\gamma(v)u+\alpha uF(w).$$
Multiplying the above equation by $g=(I-D\Delta)^{-1}u$ and integrating the result over $\Omega$, we have
\begin{equation}\label{21}
\frac{1}{2}\frac{d}{dt}\left(\int_\Omega g^2+D\int_\Omega|\nabla g|^2\right)+\frac{1}{D}\int_\Omega\gamma(v)u^2=\frac{1}{D}\int_\Omega\gamma(v)ug+\alpha \int_\Omega uF(w)g.
\end{equation}
In  view of the assumption \eqref{4} and Lemma \ref{lm2}, we get $C_1>0$ fulfilling
\begin{equation}\label{18}
\gamma(v)\leq C_1.
\end{equation}
Noticing Lemma \ref{lm3}, Lemma \ref{lm7} and the smoothing property of $F$, we get a constant $C_2>0$ such that
\begin{equation}\label{19}
F(w)\leq C_2.
\end{equation}
Substituting \eqref{18} and \eqref{19} into \eqref{21}, we obtain from Lemma \ref{lm12} that
\begin{equation}\label{22}
\frac{1}{2}\frac{d}{dt}\left(\int_\Omega g^2+D\int_\Omega|\nabla g|^2\right)+\frac{1}{D}\int_\Omega\gamma(v)u^2\leq\left(\frac{C_1}{D}+\alpha C_2\right)C_3\|g\|_{L^\infty},
\end{equation}
holds for some constant $C_3>0$. Moreover, it follows from \eqref{24} that
\begin{equation*}
D\int_\Omega|\nabla g|^2+\int_\Omega g^2=\int_\Omega ug\leq C_3\|g\|_{L^\infty}
\end{equation*}
which, added to \eqref{22} yields
\begin{equation}\label{28}
\begin{aligned}
&\frac{d}{dt}\left(\int_\Omega g^2+D\int_\Omega|\nabla g|^2\right)+\left(\int_\Omega g^2+ D\int_\Omega|\nabla g|^2\right)+\frac{2}{D}\int_\Omega\gamma(v)u^2\\
\leq&2\left(\frac{C_1}{D}+\alpha C_2+1\right)C_3\|g\|_{L^\infty}.
\end{aligned}
\end{equation}
Now we estimate the right hand side of the above inequality. For any $\max\{\frac{n}{2},1\}<p<2$, thanks to the Sobolev embedding theorem, the standard elliptic estimate and H\"{o}lder's inequality, we can find constants $C_4,C_5,C_6>0$ such that
\begin{equation*}
\begin{aligned}
\|g\|_{L^\infty}\leq &C_4\|g\|_{W^{2,p}}\leq C_5\|u\|_{L^p}\\
\leq&\frac{1}{2D}\frac{1}{\left(\frac{C_1}{D}+\alpha C_2+1\right)C_3}\int_\Omega\gamma(v)u^2+C_6\left(\int_\Omega \gamma^{-\frac{p}{2-p}}(v)\right)^{\frac{2-p}{p}}.
\end{aligned}
\end{equation*}
In view of the assumption \eqref{23}, there exist $k>0$, $\overline{v}>0$ and $C_7>0$ such that
$$v^k\gamma(v)\geq C_7\ \ \text{for all}\ \ v>\overline{v}$$
i.e.
$$\gamma^{-1}(v)\leq C^{-1}_7 v^k\ \ \text{for all}\ \ v>\overline{v}.$$
Noticing the non-increasing property of $\gamma$, we get
$$\gamma^{-1}(v)\leq\gamma^{-1}(\overline{v})\ \ \text{for all}\ \ 0\leq v\leq \overline{v}.$$
Therefore, it holds that
$$\gamma^{-1}(v)\leq\gamma^{-1}(\overline{v})+C^{-1}_7 v^k\ \ \text{for all}\ \ v\geq 0.$$
Hence, using Lemma \ref{lm1} and Lemma \ref{lm15}, there exist constants $C_8,C_9,C_{10}>0$ such that
\begin{equation}\label{26}
\begin{aligned}
\int_\Omega \gamma^{-\frac{p}{2-p}}(v)\leq&\int_{\Omega}\left(\gamma^{-1}(\overline{v})+C^{-1}_7 v^k\right)^{\frac{p}{2-p}}\\
\leq&\int_{\Omega}\left(\gamma^{-1}(\overline{v})+C^{-1}_7 \left(C_8(g+1)\right)^{k}\right)^{\frac{p}{2-p}}\\
\leq&C_9\int_{\Omega} g^{\frac{pk}{2-p}} d x+C_9\\
\leq&C_{10}
\end{aligned}
\end{equation}
which implies that
\begin{equation}\label{27}
\|g\|_{L^\infty}\leq \frac{1}{2D}\frac{1}{\left(\frac{C_1}{D}+\alpha C_2+1\right)C_3}\int_\Omega\gamma(v)u^2+C_6C_{10}^{\frac{2-p}{p}}.
\end{equation}
Combining \eqref{28}, \eqref{26} with \eqref{27}, we get
\begin{equation*}
\begin{aligned}
&\frac{d}{dt}\left(\int_\Omega g^2+D\int_\Omega|\nabla g|^2\right)+\left(\int_\Omega g^2+ D\int_\Omega|\nabla g|^2\right)+\frac{1}{D}\int_\Omega\gamma(v)u^2\\
\leq&2\left(\frac{C_1}{D}+\alpha C_2+1\right)C_3C_6C_{10}^{\frac{2-p}{p}}
\end{aligned}
\end{equation*}
which along with Gr\"{o}nwall's inequality yields a constant $C_{11}>0$ such that
$$\int_\Omega g^2+D\int_\Omega|\nabla g|^2\leq C_{11}$$
and
\begin{equation}\label{25}
\int_t^{t+\tau}\int_\Omega\gamma(v)u^2\leq C_{11}\quad\text{for all }t\in(0,T_{max}-\tau)
\end{equation}
with $\tau=\min\{1,\frac{1}{2}T_{max}\}$. Due to \eqref{27} and \eqref{25}, the following inequality
\begin{equation}\label{nie}
\int_t^{t+\tau}\int_\Omega g\leq|\Omega|\int_t^{t+\tau}\|g\|_{L^\infty}\leq C_{12}\ \ \text{for all}\ \ t\in(0,T_{max}-\tau),
\end{equation}
holds for some constant $C_{12}>0$. According to Lemma \ref{lm8}, we can find a constant $C_{13}>0$ such that
$$g_t\leq C_{13}g\ \ \text{for all}\ \ t\in(0,T_{max}).$$
Using Lemma \ref{lm9} with \eqref{nie} and the definition of $\tau$, we get a constant $C_{14}>0$ so that
$$g\leq C_{14}=\frac{C_{12}}{|\Omega|\tau}e^{C_{13}}\ \ \text{for any}\ \ (x,t)\in\Omega\times(\tau,T_{max})$$
which, along with Lemma \ref{lm8} applied to any $(x,t)\in\Omega\times[0,\tau]$, asserts that
$$g\leq C_{15}\ \ \text{for any}\ \ (x,t)\in\Omega\times[0,T_{max})$$
holds for some constant $C_{15}>0$. This completes the proof by using Lemma \ref{lm1}.
\end{proof}

\subsection{$L^\infty$-estimates of $u$}
Once we get the positive lower and upper bounds of $v$ in the previous section, then the diffusion motility function $\gamma(v)$ is neither degenerate nor singular and the estimate of $L^\infty$-norm of $u$ essentially can be derived by the procedures as shown in \cite{J-S-W2020JDE}. For completeness, we briefly demonstrate the mains steps below.
\subsubsection{The space-time \texorpdfstring{$L^2$}{L2}-bound of \texorpdfstring{$u$}{u}}
In this subsection, we aim to derive the bound of $u$ in space-time $L^2$-norm by the classical duality-based arguments (cf. \cite{J-S-W2020JDE,L-W2015CPDE,T-W2017MMMAS}).
For convenience, we introduce some notations here. Let $A_0$ denote the self-adjoint realization of $-\Delta$ defined in the Hilbert space
$$L_{\bot}^2(\Omega)=\left\{\phi\in L^2(\Omega)\,\Big|\,\int_\Omega\phi=0\right\},$$
with its domain
$$D(A_0)=\left\{\phi\in W^{2,2}(\Omega)\cap L_{\bot}^2(\Omega)\,\Big|\, \frac{\partial\phi}{\partial\nu}=0\,\text{on}\,\partial\Omega\right\}.$$
Then $A_0$ is self-adjoint and possesses bound self-adjoint fractional powers $A_0^{-\beta}$ for any $\beta>0$ (cf. \cite{M.S.1984AJM}).

Now the classical duality-based arguments lead to the boundedness of $u$ in space-time $L^2$.

\begin{lemma}\label{lm4}
Let $(u,v,w)$ be a solution of \eqref{1}. Then there exists a constant $C>0$ such that
$$\int_t^{t+\tau}\int_{\Omega}u^2\leq C\ \ \text{for all}\ \ t\in[0,T_{max}-\tau)$$
with $\tau=\min\{1,\frac{1}{2}T_{max}\}$.
\end{lemma}

\begin{proof}
According to Lemma \ref{lm2}, Lemma \ref{lm1} and \eqref{4}, we can find constants $C_1,C_2>0$ such that
$$C_1\leq \gamma(v)\leq C_2.$$
Multiplying the third equation of \eqref{1} by $\alpha$ and adding the result equation to the first equation of \eqref{1}, we get
$$(u+\alpha w)_t=\Delta(\gamma(v)u+\alpha w)$$
which can be rewritten as
\begin{equation}\label{11}
(u+\alpha w-\overline{u}-\alpha \overline{w})_t=-A_0(\gamma(v)u+\alpha w-\overline{\gamma(v)u}-\alpha \overline{w}).
\end{equation}
In view of \eqref{11} and the fact $\int_\Omega(u+\alpha w-\overline{u}-\alpha \overline{w})=0$, integrating by parts, we obtain
\begin{align*}
&\frac{1}{2}\frac{d}{dt}\int_\Omega|A_0^{-\frac{1}{2}}(u+\alpha w-\overline{u}-\alpha \overline{w})|^2\\
=&\int_\Omega A_0^{-\frac{1}{2}}(u+\alpha w-\overline{u}-\alpha \overline{w})\cdot A_0^{-\frac{1}{2}}(u+\alpha w-\overline{u}-\alpha \overline{w})_t\\
=&\int_\Omega A_0^{-1}(u+\alpha w-\overline{u}-\alpha \overline{w})\cdot(u+\alpha w-\overline{u}-\alpha \overline{w})_t\\
=&-\int_\Omega A_0^{-1}(u+\alpha w-\overline{u}-\alpha \overline{w})\cdot A_0(\gamma(v)u+\alpha w-\overline{\gamma(v)u}-\alpha \overline{w})\\
=&-\int_\Omega(u+\alpha w-\overline{u}-\alpha \overline{w})\cdot (\gamma(v)u+\alpha w-\overline{\gamma(v)u}-\alpha \overline{w})\\
=&-\int_\Omega\gamma(v)(u-\overline{u})^2-\overline{u}\int_\Omega\gamma(v)(u-\overline{u})-\alpha\int_\Omega(1+\gamma(v))(u-\overline{u})(w-\overline{w})\\
&\quad-\alpha\overline{u}\int_\Omega\gamma(v)(w-\overline{w})-\alpha^2\int_\Omega(w-\overline{w})^2\\
\leq&-C_1\int_\Omega(u-\overline{u})^2+C_2|\Omega|\overline{u}^2+2\alpha(1+C_2)|\Omega|\overline{u}\cdot\overline{w}+C_2|\Omega|\alpha\overline{u}\cdot\overline{w}-\alpha^2\int_\Omega(w-\overline{w})^2
\end{align*}
which yields a constant $C_3>0$ such that
\begin{equation}\label{29}
\begin{aligned}
\frac{d}{dt}\int_\Omega|A_0^{-\frac{1}{2}}(u+\alpha w-\overline{u}-\alpha \overline{w})|^2+2C_1\int_\Omega(u-\overline{u})^2+2\alpha^2\int_\Omega(w-\overline{w})^2\leq C_3.
\end{aligned}
\end{equation}
By the Poincar\'{e} inequality and the fact
$$\int_{\Omega}A_0^{-\frac{1}{2}}\left(u+\alpha w-\overline{u}-\alpha \overline{w}\right)=0,$$
we can find a constant $C_4>0$ such that
\begin{align*}
&\int_{\Omega}\left|A_0^{-\frac{1}{2}}(u+\alpha w-\overline{u}-\alpha \overline{w})\right|^{2} \\
\leq& C_4 \int_{\Omega}\left|\nabla A_0^{-\frac{1}{2}}(u+\alpha w-\overline{u}-\alpha \overline{w})\right|^{2} \\
=&C_4 \int_{\Omega}|u+\alpha w-\overline{u}-\alpha \overline{w}|^{2} \\
\leq& 2C_4\int_{\Omega}(u-\overline{u})^{2}+2C_4 \alpha^{2} \int_{\Omega}(w-\overline{w})^{2} \\
\leq& 2C_4\int_{\Omega}(u-\overline{u})^{2}+2C_4 \alpha^{2}|\Omega|\left\|w_{0}\right\|_{L^{\infty}}^{2}
\end{align*}
which combined with \eqref{29} implies there exists a constant $C_5>0$ such that
\begin{equation}\label{30}
\begin{aligned}
\frac{d}{dt}\int_\Omega|A_0^{-\frac{1}{2}}(u+\alpha w-\overline{u}-\alpha \overline{w})|^2&+\frac{C_1}{2C_4}\int_{\Omega}\left|A_0^{-\frac{1}{2}}(u+\alpha w-\overline{u}-\alpha \overline{w})\right|^{2}\\
&+C_1\int_\Omega(u-\overline{u})^2+2\alpha^2\int_\Omega(w-\overline{w})^2\leq C_5.
\end{aligned}
\end{equation}
An application of Gr\"{o}nwall's inequality gives a constant $C_6>0$ such that
$$\int_\Omega|A_0^{-\frac{1}{2}}(u+\alpha w-\overline{u}-\alpha \overline{w})|^2\leq C_6.$$
Integrating \eqref{30} over $(t,t+\tau)$, we get
$$\int_t^{t+\tau}\int_\Omega(u-\overline{u})^2\leq C_7$$
for some constant $C_7>0$, which implies
\begin{equation*}
\int_t^{t+\tau}\int_\Omega u^2=\int_t^{t+\tau}\int_\Omega(u-\overline{u})^2+\int_t^{t+\tau}\int_\Omega \overline{u}^2\leq C_7+\overline{u}^2|\Omega|.
\end{equation*}
Hence, we complete the proof of the lemma.
\end{proof}

\subsubsection{\texorpdfstring{$L^2$}{L2}-estimate of \texorpdfstring{$u$}{u}}


\begin{lemma}\label{lm6}
Let $(u,v,w)$ be a solution of \eqref{1}. Then there exists a constant $C>0$ such that
$$\int_{\Omega}|\nabla v|^{2}\leq C\ \ \text{for all}\ \ t\in(0,T_{max})$$
and
$$\int_t^{t+\tau}\int_{\Omega}|\Delta v|^2\leq C\ \ \text{for all}\ \ t\in[0,T_{max}-\tau).$$
\end{lemma}

\begin{proof}
By simple computations, we have
\begin{align*}
\frac{1}{2}\frac{d}{dt}\int_{\Omega}|\nabla v|^{2}=&\int_{\Omega}\nabla v\cdot\nabla v_t\\
=&\int_{\Omega}\nabla v\cdot\nabla(D\Delta v-v+u)\\
=&-D\int_{\Omega}|\Delta v|^2-\int_{\Omega}|\nabla v|^2-\int_{\Omega}u\Delta v\\
\leq&-\frac{D}{2}\int_{\Omega}|\Delta v|^2-\int_{\Omega}|\nabla v|^2+\frac{1}{2D}\int_{\Omega}u^2
\end{align*}
which leads to
$$\frac{d}{dt}\int_{\Omega}|\nabla v|^{2}+D\int_{\Omega}|\Delta v|^2+2\int_{\Omega}|\nabla v|^2\leq\frac{1}{D}\int_{\Omega}u^2.$$
An application of the Gr\"{o}nwall inequality along with Lemma \ref{lm16} and Lemma \ref{lm4} gives a constant $C_1>0$ such that
$$\int_{\Omega}|\nabla v|^{2}+\int_t^{t+\tau}\int_{\Omega}|\Delta v|^2\leq C_1.$$
Therefore, we finish the proof of this lemma.
\end{proof}

\begin{lemma}\label{lm5}
Let $(u,v,w)$ be a solution of \eqref{1}. There exist constants $c,C>0$ such that for any $p\geq2$, we have
$$\begin{aligned}
&\frac{d}{dt}\int_\Omega u^p+cp(p-1)\int_\Omega u^{p-2}|\nabla u|^2\\
\leq&Cp(p-1)\int_\Omega u^p|\nabla v|^2+Cp(p-1)\int_\Omega u^p\ \ \text{for all}\ \ t\in(0,T_{max}).
\end{aligned}$$

\end{lemma}

\begin{proof}
According to Lemma \ref{lm3}, Lemma \ref{lm7} and the hypothesis on $F$, we can find a constant $C_1>0$ such that
$$F(w)\leq C_1.$$
Noticing Lemma \ref{lm2}, Lemma \ref{lm1} and the smoothing property of $\gamma$, there exist constants $C_2, C_3>0$ such that
\begin{equation}\label{2}
\gamma(v)\geq C_2
\end{equation}
and
\begin{equation}\label{3}
\frac{|\gamma'(v)|^2}{\gamma(v)}\leq C_3.
\end{equation}
Using $u^{p-1}$ with $p\geq2$ as a test function for the first equation in \eqref{1}, integrating the resulting equation by parts and using Young's inequality, we obtain
\begin{equation*}
\begin{aligned}
\frac{1}{p}\frac{d}{dt}\int_\Omega u^p=&\int_\Omega u^{p-1}\Delta(\gamma(v)u)+\alpha\int_\Omega u^pF(w)\\
\leq&-(p-1)\int_\Omega\gamma(v) u^{p-2}|\nabla u|^2+(p-1)\int_\Omega\gamma'(v)u^{p-1}\nabla u\cdot\nabla v+C_1\alpha\int_\Omega u^p\\
\leq&-\frac{p-1}{2}\int_\Omega \gamma(v)u^{p-2}|\nabla u|^2+\frac{p-1}{2}\int_\Omega \frac{|\gamma'(v)|^2}{\gamma(v)}u^p|\nabla v|^2+C_1\alpha\int_\Omega u^p
\end{aligned}
\end{equation*}
which, combined with \eqref{2} and \eqref{3}, yields that
\begin{equation*}
\begin{aligned}
&\frac{d}{dt}\int_\Omega u^p+\frac{p(p-1)}{2}C_2\int_\Omega u^{p-2}|\nabla u|^2\\
\leq&\frac{p(p-1)}{2}C_3\int_\Omega u^p|\nabla v|^2+C_1\alpha p\int_\Omega u^p.
\end{aligned}
\end{equation*}
This finishes the proof of this lemma.
\end{proof}

Now the uniform-time boundedness of $u$ in $L^2(\Omega)$ can be established.

\begin{lemma}\label{lm13}
Let $(u,v,w)$ be a solution of \eqref{1}. Then there is a constant $C>0$ such that
$$\int_\Omega u^2\leq C\ \ \text{for all}\ \ t\in(0,T_{max}).$$
\end{lemma}

\begin{proof}
Taking $p=2$ in Lemma \ref{lm5}, we get the following estimate
\begin{equation}\label{14}
\frac{d}{dt}\int_\Omega u^2+C_1\int_\Omega |\nabla u|^2\leq C_2\int_\Omega u^2|\nabla v|^2+C_2\int_\Omega u^2
\end{equation}
for some constants $C_1,C_2>0$. Using Lemma \ref{lm6}, the Gagliardo-Nirenberg inequality and Young's inequality, we can find constants $C_3,C_4,C_5>0$ such that
\begin{equation*}
\begin{aligned}
C_2\int_\Omega u^2|\nabla v|^2\leq&C_2\|u\|_{L^4}^2\|\nabla v\|_{L^4}^2\\
\leq&C_3\left(\|\nabla u\|_{L^2}^{\frac{n}{4}}\|u\|_{L^2}^{\frac{4-n}{4}}+\|u\|_{L^2}\right)^2\left(\|\Delta v\|_{L^2}^{\frac{n}{4}}\|\nabla v\|_{L^2}^{\frac{4-n}{4}}+\|\nabla v\|_{L^2}\right)^2\\
\leq&4C_3\left(\|\nabla u\|_{L^2}^{\frac{n}{2}}\|u\|_{L^2}^{\frac{4-n}{2}}\|\Delta v\|_{L^2}^{\frac{n}{2}}\|\nabla v\|_{L^2}^{\frac{4-n}{2}}+\|\nabla u\|_{L^2}^{\frac{n}{2}}\|u\|_{L^2}^{\frac{4-n}{2}}\|\nabla v\|_{L^2}^2\right.\\
&\quad\left.+\|u\|_{L^2}^2\|\Delta v\|_{L^2}^{\frac{n}{2}}\|\nabla v\|_{L^2}^{\frac{4-n}{2}}+\|u\|_{L^2}^2\|\nabla v\|_{L^2}^2\right)\\
\leq&C_4\left(\|\nabla u\|_{L^2}^{\frac{n}{2}}\|u\|_{L^2}^{\frac{4-n}{2}}\|\Delta v\|_{L^2}^{\frac{n}{2}}+\|\nabla u\|_{L^2}^{\frac{n}{2}}\|u\|_{L^2}^{\frac{4-n}{2}}\right.\\
&\quad\left.+\|u\|_{L^2}^2\|\Delta v\|_{L^2}^{\frac{n}{2}}+\|u\|_{L^2}^2\right)\\
\leq&\frac{C_1}{2}\|\nabla u\|_{L^2}^2+C_5\left(1+\|\Delta v\|_{L^2}^{\frac{2n}{4-n}}+\|\Delta v\|_{L^2}^{\frac{n}{2}}\right)\|u\|_{L^2}^2\\
\leq&\frac{C_1}{2}\|\nabla u\|_{L^2}^2+C_5\left(1+\|\Delta v\|_{L^2}^2\right)\|u\|_{L^2}^2
\end{aligned}
\end{equation*}
which combined with \eqref{14} implies there exists a constant $C_6>0$ such that
$$\frac{d}{dt}\int_\Omega u^2+\frac{C_1}{2}\int_\Omega |\nabla u|^2\leq C_6\left(1+\|\Delta v\|_{L^2}^2\right)\int_\Omega u^2.$$
An application of Lemma \ref{lm6} and Lemma \ref{lm9} gives the desired result.
\end{proof}

\subsubsection{\texorpdfstring{$L^\infty$}{Lp}-estimate of \texorpdfstring{$u$}{u}}

\begin{lemma}\label{lm11}
Let $(u,v,w)$ be a solution of \eqref{1}. For any $p\geq1$, there exists $C>0$ such that
$$\int_\Omega|\nabla v|^p\leq C\ \ \text{for all}\ \ t\in(0,T_{max}).$$
Moreover, if $n=1$, then we can find $C>0$ such that
$$\|\nabla v\|_{L^\infty}\leq C\ \ \text{for all}\ \ t\in(0,T_{max}).$$
\end{lemma}

\begin{proof}
Applying Lemma \ref{lm10} and Lemma \ref{lm13}, the desired result is obtained.
\end{proof}

Combining Lemma \ref{lm5} and Lemma \ref{lm11}, we get the following result.

\begin{lemma}\label{lm14}
Let $(u,v,w)$ be a solution of \eqref{1}. There exists a constant $C>0$ such that for any $p\geq2$, we have
$$\frac{d}{dt}\int_\Omega u^p+p(p-1)\int_\Omega u^p\leq Cp(p-1)\left(1+p\right)^{\frac{6n}{4-n}}\left(\int_\Omega u^{\frac{p}{2}}\right)^2\ \ \text{for all}\ \ t\in(0,T_{max}).$$
\end{lemma}

\begin{proof}
From Lemma \ref{lm5}, we can find constants $C_1,C_2>0$ such that
\begin{equation}\label{15}
\frac{d}{dt}\int_\Omega u^p+C_1\frac{p-1}{p}\int_\Omega|\nabla u^{\frac{p}{2}}|^2\leq C_2p(p-1)\int_\Omega u^p|\nabla v|^2+C_2p(p-1)\int_\Omega u^p.
\end{equation}
By means of Lemma \ref{lm11} and Young's inequality, there exists a constant $C_3>0$ such that
$$C_2p(p-1)\int_\Omega u^p|\nabla v|^2\leq C_2p(p-1)\left(\int_\Omega u^{2p}\right)^{\frac{1}{2}}\left(\int_\Omega |\nabla v|^4\right)^{\frac{1}{2}}\leq C_3p(p-1)\left(\int_\Omega u^{2p}\right)^{\frac{1}{2}}$$
and
$$(C_2+1)p(p-1)\int_\Omega u^p\leq C_3p(p-1)\left(\int_\Omega u^{2p}\right)^{\frac{1}{2}},$$
which combined with \eqref{15} gives that
\begin{equation}\label{17}
\frac{d}{dt}\int_\Omega u^p+C_1\frac{p-1}{p}\int_\Omega|\nabla u^{\frac{p}{2}}|^2+p(p-1)\int_\Omega u^p\leq2C_3p(p-1)\left(\int_\Omega u^{2p}\right)^{\frac{1}{2}}.
\end{equation}
The Gagliardo-Nirenberg inequality and Young's inequality with $\varepsilon$ provide constants $C_4,C_5>0$ such that
\begin{equation}\label{16}
\begin{aligned}
2C_3\left(\int_\Omega u^{2p}\right)^{\frac{1}{2}}=&2C_3\|u^{\frac{p}{2}}\|_{L^4}^2\\
\leq& C_4\left(\|\nabla u^{\frac{p}{2}}\|_{L^2}^{\frac{3n}{2(n+2)}}\|u^{\frac{p}{2}}\|_{L^1}^{\frac{4-n}{2(n+2)}}+\|u^{\frac{p}{2}}\|_{L^1}\right)^2\\
\leq&2C_4\left(\|\nabla u^{\frac{p}{2}}\|_{L^2}^{\frac{3n}{n+2}}\|u^{\frac{p}{2}}\|_{L^1}^{\frac{4-n}{n+2}}+\|u^{\frac{p}{2}}\|^2_{1}\right)\\
\leq&C_1\frac{1}{2p^2}\|\nabla u^{\frac{p}{2}}\|^2_{L^2}+C_5\left(1+p^{\frac{6n}{4-n}}\right)\|u^{\frac{p}{2}}\|_{L^1}^2.
\end{aligned}
\end{equation}
Substituting \eqref{16} into \eqref{17} and noting $1+p^{\frac{6n}{4-n}}\leq(1+p)^{\frac{6n}{4-n}}$, we obtain
$$\frac{d}{dt}\int_\Omega u^p+p(p-1)\int_\Omega u^p\leq C_5p(p-1)\left(1+p\right)^{\frac{6n}{4-n}}\left(\int_\Omega u^{\frac{p}{2}}\right)^2.$$
Then we get the desired result.
\end{proof}

Now we can obtain the uniform-in-time boundedness of $u$ in $L^\infty(\Omega)$.

\begin{lemma}\label{lm17}
Let $(u,v,w)$ be a solution of \eqref{1}. There exists a constant $C>0$ such that
$$\|u\|_{L^\infty}\leq C\ \ \text{for all}\ \ t\in(0,T_{max}).$$
\end{lemma}

\begin{proof}
According to Lemma \ref{lm14}, there exists a constant $C_1>0$ such that for any $p\geq2$
$$\frac{d}{dt}\int_\Omega u^p+p(p-1)\int_\Omega u^p\leq C_1p(p-1)\left(1+p\right)^{\frac{6n}{4-n}}\left(\int_\Omega u^{\frac{p}{2}}\right)^2$$
which gives
\begin{equation}\label{31}
\frac{d}{dt}\left[e^{p(p-1)t}\int_\Omega u^p\right]\leqslant C_1e^{p(p-1)t}p(p-1)(1+p)^6\left(\int_\Omega u^{\frac{p}{2}}\right)^2.
\end{equation}
Integrating \eqref{31} over the time interval $[0,t]$ for $0<t<T_{max}$, we get
$$\int_\Omega u^p\leq\int_\Omega u_0^p+C_1(1+p)^6\sup_{0\leq t\leq T_{max}}\left(\int_\Omega u^{\frac{p}{2}}\right)^2.$$
Then, employing a standard Moser iteration (cf. \cite{N.D.A.1979CPDE}) or the similar argument as in \cite{T-W2013MMMAS}, the desired result can be obtained.
\end{proof}

\section{Proof of Theorem \ref{th1}}\label{sec4}


\begin{proof}[Proof of Theorem \ref{th1}]
Theorem \ref{th1} is a consequence of Lemma \ref{lm17}, Lemma \ref{lm10} and the extensibility criterion Lemma \ref{lm3}.
\end{proof}

\bigbreak

\noindent \textbf{Acknowledgement}.
The research of Z.A. Wang was supported by the Hong Kong RGC GRF grant No. 15303019 (Project ID P0030816) and an internal grant No. UAH0 (Project ID P0031504) from the Hong Kong Polytechnic University.


\end{document}